\documentclass[a4paper,12pt,final]{amsart}
\usepackage{times,a4wide,mathrsfs, nameref,showkeys,amssymb,amsmath,amsthm,enumerate,xypic,tikzsymbols,dsfont, stix}

\newcommand{\C}{\mathbb{C}}

\newcommand{\QQ}{\mathbb{Q}}

\newcommand{\PP}{\mathbb{P}}

\DeclareMathOperator{\CH}{CH}

\DeclareMathOperator{\sym}{Sym}

\newtheorem{theorem}{Theorem}[section]
\newtheorem{claim}[theorem]{Claim}

\newtheorem{proposition}[theorem]{Proposition}
\newtheorem{conjecture}[theorem]{Conjecture}

\newtheorem{convention}{Conventions}

\newtheorem{notation}[theorem]{Notation}

\newtheorem{nonumberingt}{Acknowledgements}

\begin{document}

\author[Robert Laterveer]
{Robert Laterveer}

\address{Institut de Recherche Math\'ematique Avanc\'ee,
CNRS -- Universit\'e 
de Strasbourg,\
7 Rue Ren\'e Des\-car\-tes, 67084 Strasbourg CEDEX,
FRANCE.}
\email{robert.laterveer@math.unistra.fr}

\title{Variation on a theme of Beauville--Voisin}

\begin{abstract} The Beauville--Voisin conjecture is about the Chow ring of hyper-K\"ahler varieties. We consider a variant version taking place in the ring of algebraic cycles modulo algebraic equivalence.
We prove this variant version is true in codimension at least 3 for Fano varieties of lines on cubic fourfolds, and for double EPW sextics.
\end{abstract}

\thanks{This work is supported by the Agence Nationale de la Recherche under grant ANR-20-CE40-0023.}

\maketitle

\section{Introduction}

Given a complex smooth projective variety $X$, let $A^\ast(X)=\oplus_i A^i(X)$ denote the Chow ring with $\QQ$-coefficients \cite{F}, \cite{Vo}.
It is expected that for hyper-K\"ahler varieties, the Chow ring has a particularly nice structure. Most famously, Beauville \cite{Beau} and Voisin \cite{V17} have put forth the following conjecture:

\begin{conjecture}[Beauville--Voisin \cite{Beau}, \cite{V17}]\label{conj} Let $X$ be a hyper-K\"ahler variety. The $\QQ$-algebra
\[   \langle A^1(X), c_j(T_X)\rangle\ \ \subset\ A^\ast(X)  \]
generated by divisors and Chern classes of the tangent bundle injects into cohomology, under the cycle class map.
\end{conjecture}

Conjecture \ref{conj} is true in dimension 2 \cite{BV}. In dimension larger than 2, Conjecture \ref{conj} has been verified for low-dimensional Hilbert schemes of K3 surfaces \cite{V17}, for Fano varieties of lines on cubic fourfolds \cite{V17}, for generalized Kummer varieties \cite{Fu} and for double EPW sextics \cite{L}. Otherwise, this is wide open.

A few years ago, I proposed a variant version of Conjecture \ref{conj} \cite{24.2}. This variant takes place in the ring $B^\ast(X)$ of algebraic cycles with $\QQ$-coefficients modulo {\em algebraic equivalence\/}.

\begin{conjecture}[\cite{24.2}]\label{conj2} Let $X$ be a hyper-K\"ahler variety. The $\QQ$-algebra
\[   \langle B^1(X), B^2(X), c_j(T_X)\rangle\ \ \subset\ B^\ast(X)  \]
generated by divisors, codimension 2 cycles and Chern classes of the tangent bundle injects into cohomology, under the cycle class map.
\end{conjecture}

Conjecture \ref{conj2} does not follow from Conjecture \ref{conj} as such. However, as explained in \cite[Introduction]{24.2}, Conjecture \ref{conj2} follows from the conjectural framework underlying Conjecture \ref{conj}. The idea is that for hyper-K\"ahler varieties $A^\ast(X)$ and $B^\ast(X)$ should be bigraded rings (under the intersection product), and the subrings $A^\ast_{(0)}(X)$ and $B^\ast_{(0)}(X)$ should inject into cohomology (under the cycle class map). Since it is expected that Chern classes are in $A^\ast_{(0)}(X)$, that $A^1(X)=A^1_{(0)}(X)$ and $B^2(X)=B^2_{(0)}(X)$ (this is because the deepest level of the conjectural Bloch--Beilinson filtration is supposed to be algebraically trivial, cf. \cite{Ja}), this motivates Conjectures \ref{conj} and \ref{conj2}.

The aim of this paper is to present some modest evidence for Conjecture \ref{conj2}, by looking at two of the most famous locally complete families of hyper-K\"ahler fourfolds:

\begin{theorem}[=Theorems \ref{main1} and \ref{main2}] Let $X$ be either the Fano variety of lines on a smooth cubic fourfold, or a double EPW sextic. Then Conjecture \ref{conj2} holds for $X$ in codimension $\ge 3$. Consequently, there is equality
  \[    B^2(X)\cdot B^1(X) = B^1(X)\cdot h^2\ \ \ \hbox{in}\ B^3(X)\ ,\]
  where $h\in B^1(X)$ denotes the polarization.
  \end{theorem}

In \cite[Theorem 2.1]{24.2}, I proved this theorem for Hilbert squares of K3 surfaces, and for the {\em very general\/} Fano variety of lines (the improvement to {\em all\/} Fano varieties of lines is made possible thanks to \cite{FLV}, as will be clear from the proof below). The result for double EPW sextics is entirely new, and is made possible by recent results on the Chow ring of double EPW sextics \cite{LV}, \cite{L}, \cite{Zh}. 

The codimension 2 part of Conjecture \ref{conj2} seems to remain out of reach for any hyper-K\"ahler variety of dimension larger than 2.
Indeed, it is conjectured that the Griffiths group $B^2_{hom}(X)$ vanishes for any variety $X$ with $H^3(X,\QQ)$ supported in codimension 1 (this applies in particular to hyper-K\"ahler varieties) \cite[Introduction]{Ja}. However, I don't believe there are any examples of varieties of dimension larger than 2 and with non-vanishing geometric genus for which this conjecture has been verified.

\vskip0.5cm
\begin{convention} In this paper, the word {\sl variety\/} will mean a reduced irreducible scheme of finite type over $\C$. A {\sl subvariety\/} will refer to a (possibly reducible) reduced subscheme which is equidimensional. 

{\bf All Chow groups will be with rational coefficients}: we denote by $A^i(Y)$ the Chow group of codimension $i$ cycles on $Y$ with $\QQ$-coefficients, and we write $B^i(Y)$ for the $\QQ$-vector space of codimension $i$ algebraic cycles modulo algebraic equivalence.
The notations $A^i_{hom}(Y)$ and $B^i_{hom}(Y)$ will be used to denote the subgroup of homologically trivial cycles.
For a morphism $f\colon X\to Y$, we write $\Gamma_f\in A^\ast(X\times Y)$ for the graph of $f$.

\end{convention}

\vskip0.5cm

\section{Fano varieties of lines}

\begin{notation} Let $X=F(Y)$ be the Fano variety of lines of a smooth cubic fourfold $Y\subset\PP^5$. Then $X$ is a hyper-K\"ahler variety of $K3^{[2]}$-type, and these
$X$ form a 20-dimensional locally complete family \cite{BD}.

We will write 
  \[ A^\ast(X)= A^\ast_{(\ast)}(X)\]
  for the Shen--Vial decomposition of the Chow ring \cite[Part 3]{SV} (NB: in loc. cit. this is written $\CH^\ast(X)_\ast$).
  
  It is expected that $ A^\ast_{(\ast)}(X)$ is a bigraded ring, under the intersection product (this is currently known only for the very general Fano variety of lines, cf. \cite[Theorem 3]{SV}).
  It is also expected that $A^i_{(0)}(X)$ injects into cohomology under the cycle class map; this is currently open for $i=2$.
\end{notation}

The first result of this paper is the following:

\begin{theorem}\label{main1} Let $X=F(Y)$ be the Fano variety of lines of a smooth cubic fourfold $Y\subset\PP^5$.
Let
  \[ R^\ast(X):=\bigl\langle  B^1(X), B^2(X) \bigr\rangle\ \ \subset\ B^\ast(X) \]
  be the graded $\QQ$-algebra generated by divisors and codimension 2 cycles.
  Then $R^i(X)$ injects into cohomology via the cycle class map for $i\ge 3$.
  
  Moreover, 
    \[ B^2(X)\cdot B^1(X)= B^1(X)\cdot h^2\ \ \ \hbox{in}\ B^3(X)\ ,\]
    where $h\in B^1(X)$ is the polarization.
  \end{theorem}

  \begin{proof} 
  For the very general Fano variety of lines, I already proved this result in \cite[Theorem 2.1]{24.2}.
  
The ``moreover'' part follows from the first statement combined with the ``weak Lefschetz'' isomorphism in cohomology
    \[  \cup h^2\colon\ \ H^2(X,\QQ)\ \xrightarrow{\cong}\ H^6(X,\QQ)\ .\]
    As for the first statement, observe that this is trivially true for $i=4$ (algebraic and homological equivalence coincide for zero-cycles), and so one only needs to treat the case $i=3$.  
       
    To prove the first statement for $i=3$, we will use the Shen--Vial bigrading of the Chow ring of $X$. There is an induced bigrading of $B^\ast(X)$, which can be defined as
      \[  B^i_{(j)}(X):= (\pi^{2i-j}_X)_\ast B^i(X)\ ,\]
      where $\{\pi^\ast_X\}$ is the Chow--K\"unneth decomposition constructed in \cite{SV}. By construction, the subgroup $B^i_{(j)}(X)$ is the image of $A^i_{(j)}(X)$ under the
      natural surjective map $A^\ast(X)\to B^\ast(X)$. We now claim the following:
        
    \begin{claim} $B^2(X)=B^2_{(0)}(X)$.
       \end{claim}
  
The claim implies the theorem. Indeed, it is known that
   \[  A^1(X)\cdot A^2_{(0)}(X)\ \subset \ A^3_{(0)}(X)\  \]
(\cite[Proposition A.7]{FLV}; this is the crucial new ingredient that wasn't available at the time I wrote \cite{24.2}), and so a fortiori 
 \[ B^1(X)\cdot B^2(X)=    B^1(X)\cdot B^2_{(0)}(X)\ \subset \ B^3_{(0)}(X)\ . \]  
 But $A^3_{(0)}(X)$ is known to inject into cohomology under the cycle class map (this is \cite[Theorem 3.3]{SV}; one could also cite \cite[Lemma 2.3]{24.2} or \cite[Theorems 1.7 and 1.9]{Kre}), and so a fortiori $B^3_{(0)}(X)$ injects into cohomology.
 
 It remains to prove the claim. This was already done in \cite[Claim 2.2]{24.2}; let us briefly recall the argument. The point is that 
   \[ A^2(X)= A^2_{(0)}(X)\oplus A^2_{(2)}(X)\ ,\]
   and so one just needs to prove that $A^2_{(2)}(X)$ consists of algebraically trivial cycles. But
 Shen--Vial have proven \cite[Theorems 2.2 and 2.4]{SV} that there is a correspondence $L\in A^2(X\times X)$ such that
   \[  A^2_{(2)}(X) = L_\ast A^4_{(2)}(X)\ .\]
   Now, zero-cycles in $A^4_{(2)}(X)$ are (homologically trivial and hence) algebraically trivial. Since algebraic equivalence is an adequate equivalence relation, it is preserved by correspondences and hence $A^2_{(2)}(X)$ consists of algebraically trivial cycles. The claim, and the theorem, are now proven.
  \end{proof}

\vskip0.5cm

\section{Double EPW sextics}

\begin{notation} Double EPW sextics have been constructed by O'Grady \cite{OG2}, \cite{OG4}, \cite{OG5} as double covers $X\to Z$ of certain sextic hypersurfaces $Z\subset\PP^5$.
They form a locally complete family of hyper-K\"ahler varieties of $K3^{[2]}$-type. By construction, they come with a covering involution $\iota$ which is anti-symplectic.
\end{notation}

Here is the second result of this paper:

\begin{theorem}\label{main2} Let $X$ be a double EPW sextic with covering involution $\iota$. Let
  \[ R^\ast(X):=\bigl\langle  B^1(X), B^2(X) \bigr\rangle\ \ \subset\ B^\ast(X) \]
  be the graded $\QQ$-algebra generated by divisors and codimension 2 cycles.
  Then $R^i(X)$ injects into cohomology via the cycle class map for $i\ge 3$.
  
  Moreover, 
    \[ R^3(X)= B^1(X)\cdot h^2\ ,\]
    where $h\in B^1(X)$ is the polarization.
   \end{theorem}
  
  \begin{proof} The ``moreover'' part follows from the first statement combined with the ``weak Lefschetz'' isomorphism in cohomology
    \[  \cup h^2\colon\ \ H^2(X,\QQ)\ \xrightarrow{\cong}\ H^6(X,\QQ)\ .\]
    
    To prove the first statement, we make the following claim:
    
    \begin{claim} 
    There is equality 
      \[ B^2(X)= B^2(X)^+ \oplus B^1(X)^+\cdot B^1(X)^-\ ,\]
      where we write $B^\ast(X)^+$ and $B^\ast(X)^-$ for the $\iota$-invariant, resp. $\iota$-anti-invariant part of $B^\ast(X)$.
    \end{claim}
    
    The claim implies the theorem, since it is known that the subgroup of the Chow group
    \[ \langle A^1(X), A^2(X)^+, c_j(T_X)\rangle\cap A^3(X) \]
    injects into cohomology, under the cycle class map \cite[Theorem 3.1]{L}.
    
  To prove the claim, let us invoke the following equality of correspondences, involving the projector on the $\iota$-anti-invariant part $\Delta_X^-:={1\over 2}(\Delta_X-\Gamma_\iota)$:
  
  \begin{proposition}[\cite{L}]\label{deltaminus}

Let $X$ be a smooth double EPW sextic, and let $\iota$ denote its covering involution. There is a relation

\[
\Delta_X^-= F +  I\cdot G \ \ \hbox{in}\  B^4(X\times X)\ ,
\]
with the following properties:
\begin{enumerate}
\item $F$ and $G$ are decomposable correspondences, i.e.

$$F\ ,\ G\ \ \in \langle (p_1)^*B^*_{}(X), (p_2)^*B^\ast_{}(X) \rangle\ $$
where $p_1, p_2$ denote projections to the first resp. second factor\, ;
%
%

\item $I$ is a codimension 2 cycle, i.e.
\[ I  \in \ \  B^2_{}(X\times X)\ .\]
%
\end{enumerate}

    \end{proposition}
    
\begin{proof} This is \cite[Equality (8) in Proof of Proposition 3.2]{L}, which is based on work of Zhang \cite{Zh}.
    \end{proof}

Let us now prove the claim. 
Given $\gamma\in B^2(X)$, we may decompose
  \[ \gamma= \gamma^+ + \gamma^-\ \ \hbox{in}\ B^2(X) \ ,\]
  where $\gamma^+\in B^2(X)^+$ is $\iota$-invariant and $\gamma^-\in B^2(X)^-$ is  $\iota$-anti-invariant.
  
  We observe that there is an isomorphism
     \[ \cup h\colon \ H^2(X,\QQ)^- \ \xrightarrow{\cong}\  H^4(X,\QQ)^- \]
  (indeed, $H^4(X,\QQ)=\sym^2 H^2(X,\QQ)$ and $H^2(X,\QQ)^+=\QQ[h]$), and so there exists a divisor $D\in B^1(X)^-$ such that $\gamma^-=h\cdot D$ in cohomology, i.e.
   \[ \gamma^- -h\cdot D\ \ \in B^2_{hom}(X)^-\ .\]  
  
    Applying the equality of correspondences of Proposition \ref{deltaminus} to the cycle class $ \gamma^- -h\cdot D \in B^2_{hom}(X)^- $, we obtain an equality 
   \begin{equation}\label{act}  \gamma^- - h\cdot D = \bigl( F +  I\cdot G\bigr){}_\ast (\gamma^- -h\cdot D)\ \ \hbox{in}\ B^2(X)\ .\end{equation}
Let us now check that the right-hand side of \eqref{act} is zero. First, the correspondence $F$, being decomposable, acts as zero on $B^\ast_{hom}(X)$ (this is because the action of a decomposable correspondence factors over $B^\ast_{hom}(point)=0$). 

Next, the correspondence $ G$ can be written as
  \[  G = (p_1)^\ast(a) + \sum_i (p_1)^\ast(E_i)\cdot (p_2)^\ast(E_i^\prime) + (p_2)^\ast(a^\prime)\ \ \hbox{in}\ B^2(X\times X)\ ,\]
  where $a,a^\prime\in B^2(X)$ and $E_i, E_i^\prime\in B^1(X)$.
  We thus need to prove that
  \begin{equation}\label{this}  \Bigl( I\cdot (p_1)^\ast(a) + I\cdot \sum_i (p_1)^\ast(E_i)\cdot (p_2)^\ast(E_i^\prime) + I\cdot (p_2)^\ast(a^\prime) \Bigr){}_\ast (  \gamma^- -h\cdot D )=0 \ .\end{equation}  
   
We will establish this in 3 steps, starting with the piece of the form $I\cdot (p_2)^\ast(a^\prime)$. Applying the projection formula, we find equality
  \[  \bigl(    I\cdot (p_2)^\ast(a^\prime)  \bigr){}_\ast (\delta)  =    a^\prime\cdot \bigl( I_\ast(\delta)\bigr)\ \ \hbox{in}\ B^2(X)\ ,\ \ \forall \delta\in B^2_{hom}(X)\ .\]
But $I_\ast(\delta)$ is in $B^0_{hom}(X)=0$ and so 
  \[  \bigl(    I\cdot (p_2)^\ast(a^\prime)  \bigr){}_\ast (\delta)  =   0\ \ \hbox{in}\ B^2(X)\ \ \ \forall \delta\in B^2_{hom}(X)\ .\]
 
  Next, let us consider the action of correspondences of the form $I\cdot (p_1)^\ast(E_i)\cdot (p_2)^\ast(E_i^\prime)$. Again using the projection formula, we find that
  \[            \bigl(    I\cdot (p_1)^\ast(E_i)\cdot (p_2)^\ast(E_i^\prime)  \bigr){}_\ast (\delta)  =    E_i^\prime\cdot \bigl( I_\ast(\gamma\cdot E_i)\bigr)\ \ \hbox{in}\ B^2(X)\ ,\ \ \forall \delta\in B^2_{hom}(X)\ .\]
  Noting that $ I_\ast(\delta\cdot E_i)\in B^1_{hom}(X)=0$ we conclude that
       \[            \bigl(    I\cdot (p_1)^\ast(E_i)\cdot (p_2)^\ast(E_i^\prime)  \bigr){}_\ast (\delta)  = 0  \ \ \forall   \delta\in B^2_{hom}(X)\ .\]     
    
 Finally, let us analyze the action of $I\cdot (p_1)^\ast(a)$.   
      Using the projection formula, we find
    \[    \bigl(I\cdot (p_1)^\ast(a)\bigr){}_\ast (\delta)=      
                                                                     I_\ast  ( \delta\cdot a )\ .\]
   But $\delta\cdot a\in B^4_{hom}(X)=0$ (algebraic and homological equivalence coincide for zero-cycles), and so
   \[  \bigl(I\cdot (p_1)^\ast(a)\bigr){}_\ast (\delta)= 0\ \ \hbox{in}\ B^2(X)\ \ \forall\ \delta\in B^2_{hom}(X)\ .\]
    Equality \eqref{this}, and hence the theorem, are proven.                                                              
        \end{proof}

%
%
%
%

\vskip1cm

 \begin{nonumberingt} Thanks to the referee for pertinent and helpful comments.
Many thanks to Kai for his impeccable and beautiful playing of Beethoven's Rondo Op. 51 no. 1 $\twonotes\twonotes$
\end{nonumberingt}

\end{document}